\documentclass[a4paper,11pt,twoside]{article}

\usepackage{amsmath,amssymb,amsthm,mathrsfs}

\usepackage{hyperref}
\usepackage[left=1.45in,right=1.45in]{geometry}

\swapnumbers

\numberwithin{equation}{section}

\newtheorem{Thm}{Theorem}[section]
\newtheorem{Prop}[Thm]{Proposition}
\newtheorem{Lem}[Thm]{Lemma}
\theoremstyle{definition}
\newtheorem{Rem}[Thm]{Remark}
\newtheorem{Expl}[Thm]{Example}

\newcommand{\R}{\mathbb{R}}
\newcommand{\Z}{\mathbb{Z}}

\newcommand{\E}{\operatorname{E}}
\newcommand{\e}{{\rm e}}

\newcommand{\cA}{\mathscr{A}}

\newcommand{\cF}{\mathscr{F}}
\newcommand{\cK}{\mathscr{K}}

\renewcommand{\b}{\bar}

\newcommand{\comb}{\text{\rm comb}}
\newcommand{\CAT}{\operatorname{CAT}}
\newcommand{\del}{\delta}
\newcommand{\Del}{\Delta}
\newcommand{\di}{\partial}
\newcommand{\diam}{\operatorname{diam}}
\newcommand{\eps}{\varepsilon}
\newcommand{\es}{\emptyset}
\newcommand{\Gam}{\Gamma}
\newcommand{\gam}{\gamma}

\newcommand{\lam}{\lambda}
\newcommand{\id}{\operatorname{id}}

\newcommand{\olX}{\,\overline{\!X}}
\newcommand{\olR}{\,\overline{\!R}}
\renewcommand{\rho}{\varrho}
\newcommand{\rk}{\operatorname{rk}}
\newcommand{\sig}{\sigma}

\newcommand{\sm}{\setminus}
\newcommand{\spt}{\operatorname{spt}}
\newcommand{\sub}{\subset}

\title{Convex geodesic bicombings and hyperbolicity}

\author{Dominic Descombes \& Urs Lang}

\date{April 19, 2014}

\begin{document}


\maketitle

\begin{abstract}
A geodesic bicombing on a metric space selects for every pair of points a 
geodesic connecting them. We prove existence and uniqueness results for 
geodesic bicombings satisfying different convexity conditions.
In combination with recent work by the second author on injective hulls, 
this shows that every word hyperbolic group acts geometrically on a proper,
finite dimensional space $X$ with a unique (hence equivariant) convex geodesic
bicombing of the strongest type. Furthermore, the Gromov boundary of~$X$
is a $Z$-set in the closure of~$X$, and the latter is a metrizable absolute 
retract, in analogy with the Bestvina--Mess theorem on the Rips complex. 
\end{abstract}
 

\section{Introduction}

By a {\em geodesic bicombing\/} $\sig$ on a metric space $(X,d)$ we mean a map 
\[
\sig \colon X \times X \times [0,1] \to  X
\]
that selects for every pair $(x,y) \in X \times X$ a constant speed geodesic
$\sig_{xy} := \sig(x,y,\cdot)$ from $x$ to $y$ (that is, 
$d(\sig_{xy}(t),\sig_{xy}(t')) = |t - t'|d(x,y)$ for all $t,t' \in [0,1]$, 
and $\sig_{xy}(0) = x$, $\sig_{xy}(1) = y$). 
We are interested in geodesic bicombings satisfying one of the 
following two conditions, 
each of which implies that $\sig$ is continuous and, hence, $X$ 
is contractible. We call
$\sig$ {\em convex\/} if
\begin{equation} \label{eq:convex}
\text{the function $t \mapsto d(\sig_{xy}(t),\sig_{x'y'}(t))$ is 
convex on $[0,1]$}
\end{equation}
for all $x,y,x',y' \in X$, and we say that $\sig$ is {\em conical\/} if 
\begin{equation} \label{eq:conical}
d(\sig_{xy}(t),\sig_{x'y'}(t)) \le (1-t)\,d(x,x') + t\,d(y,y')
\quad \text{for all $t \in [0,1]$} 
\end{equation}
and $x,y,x',y' \in X$. 
Obviously every convex geodesic bicombing is conical, but the reverse 
implication does not hold in general (see Example~\ref{Expl:not-convex}). 
For this, in order to pass condition~\eqref{eq:conical} to subsegments,
one would need to know in addition that $\sig$ is {\em consistent\/} in the 
sense that 
\begin{equation} \label{eq:consistent}
\sig_{pq}(\lam) = \sig_{xy}((1-\lam)s + \lam t)
\end{equation} 
whenever $x,y \in X$, $0 \le s \le t \le 1$, 
$p := \sig_{xy}(s)$, $q := \sig_{xy}(t)$, and $\lam \in [0,1]$.
Thus, every conical and consistent geodesic bicombing is convex.
Furthermore, we will say that $\sig$ is {\em reversible\/} if 
\begin{equation} \label{eq:reversible}
\sig_{xy}(t) = \sig_{yx}(1-t) \quad \text{for all $t \in [0,1]$} 
\end{equation}
and $x,y \in X$. This property is not imposed here but will 
be obtained at no extra cost later.
Basic examples of convex geodesic bicombings are given by the linear 
geodesics $\sig_{xy}(t) := (1-t)x + ty$ in a linearly convex subset $X$
of a normed space or by the unique geodesics
$\sig_{xy} \colon [0,1] \to X$ in a $\CAT(0)$ space~\cite{Bal,BriH} or 
a Busemann space~\cite{AleB,Pap} (where $t \mapsto d(\sig(t),\tau(t))$ 
is convex for {\em every\/} pair of geodesics 
$\sig,\tau \colon [0,1] \to X$.)  
An additional source of conical geodesic bicombings is the fact that 
this weaker notion behaves well under $1$-Lipschitz retractions 
(see Lemma~\ref{Lem:retr}). 
Consistent and reversible geodesic bicombings have also been 
employed in~\cite{HitL, FoeL}. 

In view of the primary examples, the existence of a convex or conical 
geodesic bicombing on a metric space may be regarded as a weak 
(but non-coarse) global nonpositive curvature condition. One thus arrives 
at the following hierarchy of properties 
$\text{(A)} \Rightarrow \text{(B)} \Rightarrow \text{(C)} \Rightarrow 
\text{(D)} \Rightarrow \text{(E)}$
for a geodesic metric space $X$:
\begin{itemize}
\item[(A)] $X$ is a $\CAT(0)$ space;
\item[(B)] $X$ is a Busemann space;
\item[(C)] $X$ admits a convex and consistent geodesic bicombing;
\item[(D)] $X$ admits a convex geodesic bicombing;
\item[(E)] $X$ admits a conical geodesic bicombing.
\end{itemize}
Clearly, if $X$ is uniquely geodesic,
then $\text{(E)} \Rightarrow \text{(B)}$.
When $X$ is a normed real vector space, (A)~holds if and only if the norm 
is induced by an inner product (\cite[Proposition~II.1.14]{BriH}), 
(B)~holds if and only if the norm is strictly convex 
(\cite[Proposition~8.1.6]{Pap}), and (C)~is always satisfied. 
In the general case, (C)~is stable under limit 
operations, whereas~(B) is not (compare Sect.~10 in~\cite{Kle}).
It is unclear whether $\text{(E)} \Rightarrow \text{(D)}$ and 
$\text{(D)} \Rightarrow \text{(C)}$ without further conditions. 
One of the purposes of this paper is to establish these implications 
under suitable assumptions, and to address questions of uniqueness.
Our first result refers to~(E) and~(D).

\begin{Thm} \label{Thm:intro-1}
Let $X$ be a proper metric space with a conical geodesic bicombing. 
Then $X$ also admits a convex geodesic bicombing. 
\end{Thm}

The idea of the proof is to resort to a discretized convexity condition
and then to gradually decrease the parameter of discreteness by 
the ``cat's cradle'' construction from~\cite{AleB}. The passage 
from~(D) to~(C) appears to be more subtle.
We first observe that the linear segments 
$t \mapsto (1-t)x + ty$ in an arbitrary normed space $X$ may be characterized 
as the curves $\gam \colon [0,1] \to X$ with the property that
$t \mapsto d(z,\gam(t))$ is convex on $[0,1]$
for every $z \in X$ (see Theorem~\ref{Thm:straight-in-vs}). 
We therefore term curves with this property {\em straight\/} 
(in a metric space $X$). 
Since the geodesics of a convex bicombing are necessarily 
straight, the above observation shows that normed spaces have no such 
bicombing other than the linear one. By contrast, there are instances of
non-linear straight geodesics in compact convex subspaces of normed spaces
as well as multiple convex geodesic bicombings in compact metric spaces 
(see Examples~\ref{Expl:nonlinear-straight} and~\ref{Expl:multiple-convex}). 
Nevertheless, we prove a strong uniqueness property of straight curves 
(Proposition~\ref{Prop:unique-straight}) in spaces satisfying a certain 
finite dimensionality assumption, introduced by Dress in~\cite{Dre} 
and explained further below. 
This gives the following result regarding items~(D) and~(C).

\begin{Thm} \label{Thm:intro-2}
Let $X$ be a metric space of finite combinatorial dimension in the sense 
of Dress, or with the property that every bounded subset has finite 
combinatorial dimension. 
Suppose that $X$ possesses a convex geodesic bicombing $\sig$.
Then $\sig$ is consistent, reversible, and unique, that is,
$\sig$ is the only convex geodesic bicombing on $X$. 
\end{Thm}

Our interest in Theorems~\ref{Thm:intro-1} and~\ref{Thm:intro-2}
comes from the fact that property~(E) holds in particular for all
absolute $1$-Lipschitz retracts $X$. (To see this, embed $X$ isometrically 
into $\ell_\infty(X)$ and retract the linear geodesics to $X$; compare 
again Lemma~\ref{Lem:retr}). These correspond to the injective 
objects in the category of metric spaces and $1$-Lipschitz maps:
$X$ is {\em injective\/} if for every isometric inclusion 
$A \subset B$ of metric spaces and every $1$-Lipschitz map 
$f\colon A \to X$ there exists a $1$-Lipschitz extension 
$\bar f \colon B \to X$ of $f$. Basic examples include the real line, 
all $\ell_\infty$ spaces, and all metric ($\R$-)trees. Furthermore,
by a 50-year-old result of Isbell~\cite{Isb}, every metric space $X$ 
possesses an essentially unique {\em injective hull\/} $\E(X)$. 
This yields a compact metric space if $X$ is compact and a finite polyhedral 
complex with $\ell_\infty$ metrics on the cells in case $X$ is finite. 
Isbell's explicit construction was rediscovered and further explored 
by Dress~\cite{Dre}, who gave it the name {\em tight span}. 
The {\em combinatorial dimension\/} $\dim_\comb(X)$ of a general metric 
space $X$ is the supremum of the dimensions of the polyhedral complexes 
$\E(Y)$ for all finite subspaces~$Y$ of~$X$. In case $X$ is already injective, 
every such $\E(Y)$ embeds isometrically into $X$, so $\dim_\comb(X)$ is bounded
by the supremum of the topological dimensions of compact subsets of $X$.
Dress also gave a characterization of the condition $\dim_\comb \le n$
in terms of a $2(n+1)$-point inequality. We restate his result in 
Theorem~\ref{Thm:dress} and provide a streamlined proof in an appendix to 
this paper. 

In~\cite{L}, the second author proved that if $\Gamma$ is a Gromov hyperbolic 
group, endowed with the word metric with respect to some finite generating set,
then the injective hull $\E(\Gamma)$ is a proper, finite-dimensional 
polyhedral complex with finitely many isometry types of cells,
isometric to polytopes in $\ell_\infty$ spaces,  
and $\Gamma$ acts properly and cocompactly on $\E(\Gamma)$ by cellular 
isometries. Furthermore, after barycentric subdivision, the resulting 
simplicial $\Gamma$-complex is a model for the classifying 
space $\underbar{E}\Gamma$ for proper actions. 
Since $X = \E(\Gamma)$ satisfies property~(E),
Theorems~\ref{Thm:intro-1} and~\ref{Thm:intro-2} together now show that 
$\E(\Gamma)$ possesses a convex geodesic bicombing that is consistent,
reversible, and unique, hence equivariant with respect to the full 
isometry group of $\E(\Gamma)$. 
(The existence of an {\em equivariant conical\/} 
geodesic bicombing on $\E(\Gamma)$ was known before, 
see Proposition~3.8 in~\cite{L}, but the proof 
of that fact did not give any indication on the consistency and uniqueness
of the bicombing.) In particular, the following holds.

\begin{Thm} \label{Thm:intro-3}
Every word hyperbolic group $\Gamma$ acts properly and cocompactly by 
isometries on a proper, finite-dimensional metric space $X$ with a 
convex geodesic bicombing that is furthermore consistent, reversible, 
and unique, hence equivariant with respect to the isometry group of $X$.
\end{Thm}

In a last part of the paper, we discuss the asymptotic geometry of 
complete metric spaces $X$ with a convex and consistent geodesic 
bicombing $\sig$. We define the geometric boundary $\di_\sig X$ 
in terms of geodesic rays consistent with $\sig$, and we 
equip $\olX_\sig = X \cup \di_\sig X$ with a natural metrizable 
topology akin to the cone topology in the case of $\CAT(0)$ spaces.
In view of Theorem~\ref{Thm:intro-3}, this unifies and generalizes the 
respective constructions for $\CAT(0)$ or Busemann spaces and for
hyperbolic groups. By virtue of the bicombing one can then give a rather
direct proof of the following result on~$\olX_\sig$. Recall that a metrizable 
space is an {\em absolute retract\/} if it is a retract of every metrizable 
space containing it as a closed subspace.

\begin{Thm} \label{Thm:intro-4}
Let $X$ be a complete metric space with a convex and consistent geodesic 
bicombing $\sig$. Then $\olX_\sig$ is contractible, locally 
contractible, and an absolute retract.
Moreover, $\di_\sig X$ is a $Z$-set in $\olX_\sig$, 
that is, for every open set $U$ in $\olX_\sig$ the inclusion 
$U \setminus \di_\sig X \hookrightarrow U$ is a homotopy equivalence.
\end{Thm}

Note that we do not assume $X$ to be proper or finite dimensional.
Bestvina and Mess~\cite{BesM} proved the analogous result for the Gromov 
closure $\overline{P(\Gam)}$ of the (contractible) Rips complex of a 
hyperbolic group $\Gam$. Theorem~\ref{Thm:intro-3} and 
Theorem~\ref{Thm:intro-4} together thus provide an alternative to their 
result, which has a number of important applications 
(see Corollaries~1.3 and~1.4 in~\cite{BesM}). 


\section{From conical to convex bicombings}

In this section, we first discuss some simple properties and examples
of conical geodesic bicombings, as defined in~\eqref{eq:conical}, 
then we prove Theorem~\ref{Thm:intro-1}.

\begin{Lem} \label{Lem:retr}
Let $\b X$ be a metric space with a conical geodesic bicombing
$\b\sig$. If $\pi \colon \b X \to X$ is a $1$-Lipschitz retraction
onto some subspace $X$ of $\b X$, then 
$\sig := \pi \circ \b\sig|_{X \times X \times [0,1]}$ 
defines a conical geodesic bicombing on $X$.
\end{Lem}

\begin{proof}
Note that since $\pi$ is a $1$-Lipschitz retraction, $\sig$ is indeed
a geodesic bicombing on $X$. Furthermore, since $\pi$ is
$1$-Lipschitz and $\b\sig$ is conical, we have
\[
d(\sig_{xy}(t),\sig_{x'y'}(t)) \le d(\b\sig_{xy}(t),\b\sig_{x'y'}(t))
\le (1-t)\,d(x,x') + t\,d(y,y')
\]
for all $t \in [0,1]$ and $x,y,x',y' \in X$, so $\sig$ is conical as
well.
\end{proof}

As mentioned earlier, a direct consequence of this lemma is that
all injective metric spaces (or absolute $1$-Lipschitz retracts) admit
conical geodesic bicombings. An equally simple application gives 
an example of a conical geodesic bicombing that is not convex.

\begin{Expl} \label{Expl:not-convex}
Consider the set $X$ of all $(u,v) \in \R^2$ with $|u|
\le 2$ and $b(u) := |u| - 1 \le v \le |b(u)|$, endowed with the metric
induced by the $\ell_\infty$ norm on $\R^2$. Note that this is a
geodesic metric space. With respect to this metric, 
the vertical retraction $\pi$ from the triangle
$\b X := \{(u,v) : b(u) \le v \le 1\}$ onto $X$ that maps
$(u,v)$ to $(u,\min\{v,|b(u)|\})$ is $1$-Lipschitz. The linear
geodesic bicombing $\b\sig$ on $\b X$, defined by $\b\sig_{xy}(t) :=
(1-t)x + ty$, is convex. By Lemma~\ref{Lem:retr}, $\sig := \pi \circ
\b\sig|_{X \times X \times [0,1]}$ defines a conical geodesic
bicombing on $X$. For $x = (-2,1)$ and $y = (2,1)$, we have
$\sig_{xy}(1/4) = (-1,0)$, $\sig_{xy}(3/4) = (1,0)$, and
$\sig_{xy}(1/2) = (0,1)$. Hence, for $z = (0,-1)$, the function $t
\mapsto \|\sig_{xy}(t) - \sig_{zz}(t)\|_\infty = \|\sig_{xy}(t) - z\|_\infty$ is
clearly not convex.
\end{Expl}

We now define a relaxed notion of convexity that will be useful 
for the proof of Theorem~\ref{Thm:intro-1}.
We say that a geodesic bicombing $\sig$ on a metric space $X$ is 
{\em $1/n$-discretely convex\/} if, for all $x,y,x',y' \in X$, 
the convexity condition
\[
d(\sig_{xy}(s),\sig_{x'y'}(s)) \le
\frac{t-s}{t-r} \,d(\sig_{xy}(r),\sig_{x'y'}(r)) 
+ \frac{s-r}{t-r} \,d(\sig_{xy}(t),\sig_{x'y'}(t))
\]
holds whenever the three numbers $r < s < t$
belong to $[0,1] \cap (1/n)\Z$. To check this condition it clearly
suffices to verify the ``local'' inequality
\begin{align*}
2d(\sig_{xy}(s),\sig_{x'y'}(s)) 
&\le d(\sig_{xy}(s-1/n),\sig_{x'y'}(s-1/n)) \\
&\quad + d(\sig_{xy}(s+1/n),\sig_{x'y'}(s+1/n))
\end{align*}
for all $s \in (0,1) \cap (1/n)\Z$. Note that every conical geodesic
bicombing is $1/2$-discretely convex.

\begin{Prop} \label{Prop:discr-convex}
Suppose that $X$ is a complete metric space with a geodesic bicombing
$\sig$ that is conical and $1/n$-discretely convex for some integer 
$n \ge 2$. Then $X$ also admits a geodesic bicombing that is conical and
$1/(2n-1)$-discretely convex.
\end{Prop}

\begin{proof}
Set $m := 2n-1$. To construct the desired
new bicombing $\tilde{\sig}$, fix $x$ and $y$ and define the sequences
$p_i$ and $q_i$ recursively as $q_0 := \sig_{xy}(n/m)$ and
\[
p_i := \sig_{xq_{i-1}}(1 - 1/n), \quad q_i := \sig_{p_iy}(1/n),
\quad \text{for $i \ge 1$.} 
\]
Since $\sigma$ is conical, we get the inequalities
\begin{align*}
d(p_i,p_{i+1}) &\le (1-1/n) \,d(q_{i-1},q_i), \\
d(q_i,q_{i+1}) &\le (1-1/n) \,d(p_i,p_{i+1}).
\end{align*}
It follows that $p_i$ and $q_i$ are Cauchy sequences, so $p_i \to p$
and $q_i \to q$. Then $\sig_{xq_i} \to \sig_{xq}$ and $\sig_{p_iy} \to \sig_{py}$ 
uniformly, again because $\sig$ is conical. 
Note that $p = \sig_{xq}(1 - 1/n)$ and $q = \sig_{py}(1/n)$;
we can thus define $\tilde{\sig}_{xy}(s)$ for $s \in [0,1] \cap (1/m)\Z$ 
so that 
\begin{equation} \label{eq:tilde-sig}
\tilde{\sig}_{xy}(s) =
\begin{cases}
\sigma_{xq}\bigl((m/n)s\bigr) & \text{if $s \le n/m$,} \\
\sigma_{py}\bigl((m/n)s -(n-1)/n\bigr) & \text{if $s \ge (n-1)/m$.}
\end{cases}
\end{equation}
To declare $\tilde{\sig}_{xy}$ on all of $[0,1]$, we connect any pair 
of consecutive points
$x' := \tilde{\sig}_{xy}(s)$ and $y' := \tilde{\sig}_{xy}(s+1/m)$, where 
$s \in [0,1) \cap (1/m)\Z$, by the geodesic $t \mapsto \sig_{x'y'}(m(t-s))$ for  
$t \in [s,s + 1/m]$. 

Now if $p',q'$ and $\tilde{\sig}_{x'y'}$ result from the same 
construction for two points $x'$ and~$y'$, we want to show that for
$s \in (0,1) \cap (1/m)\Z$ we have
\begin{align*}
2d(\tilde{\sig}_{xy}(s),\tilde{\sig}_{x'y'}(s)) 
&\le d(\tilde{\sig}_{xy}(s-1/m),\tilde{\sig}_{x'y'}(s-1/m)) \\
&\quad + d(\tilde{\sig}_{xy}(s+1/m),\tilde{\sig}_{x'y'}(s+1/m)).
\end{align*} 
In view of~\eqref{eq:tilde-sig} this corresponds to the inequality
\[
2d(\tau(t),\tau'(t)) 
\le d(\tau(t-1/n),\tau'(t-1/n)) 
+ d(\tau(t+1/n),\tau'(t+1/n))
\]
where $\tau = \sig_{xq}$, $\tau' = \sig_{x'q'}$, $t = (m/n)s$ if
$s \le (n-1)/m$ and $\tau = \sig_{py}$, $\tau' = \sig_{p'y'}$, 
$t = (m/n)s -(n-1)/n$ if $s \ge n/m$. However, these inequalities
hold since $\sig$ is $1/n$-discretely convex. This shows that 
$\tilde{\sig}$ is $1/m$-discretely convex. Now it follows easily from the 
construction that $\tilde{\sig}$ is also conical.
\end{proof}

From this result, we now obtain Theorem~\ref{Thm:intro-1}
by an application of the Arzel\`a--Ascoli theorem, which requires $X$ to be
proper. We do not know whether the implication 
$\text{(E)} \Rightarrow \text{(D)}$ holds in general 
without this assumption.

\begin{proof}[Proof of Theorem~\ref{Thm:intro-1}]
Starting from the given conical geodesic bicombing $\sig^1 := \sig$,
we construct, by means of Proposition~\ref{Prop:discr-convex}, 
a sequence of conical geodesic
bicombings $\sig^k$ on $X$ such that $\sig^k$ is $1/n_k$-discretely
convex, where $n_1 = 2$ and $n_{k+1} = 2n_k - 1$.
This collection of maps $\sig^k$ is equicontinuous on every bounded set,
and for every fixed $(x,y,t)$ in the separable domain $X \times
X \times [0,1]$, the sequence $\sig^k_{xy}(t)$ remains in a compact subset 
of $X$. One may thus extract a subsequence
$\sig^{k(l)}$ that converges uniformly on every compact
set to a map $\b\sig$, which is clearly a geodesic bicombing. Convexity
\[
d(\b\sig_{xy}(s),\b\sig_{x'y'}(s))
\le \frac{t-s}{t-r} \,d(\b\sig_{xy}(r),\b\sig_{x'y'}(r)) 
+ \frac{s-r}{t-r} \,d(\b\sig_{xy}(t),\b\sig_{x'y'}(t)),
\] 
where $r < s < t$, follows from the corresponding inequality
for $\sig^{k(l)}_{xy},\sig^{k(l)}_{x'y'}$ and $r_l < s_l < t_l$ by choosing these
numbers in $[0,1] \cap (1/n_{k(l)})\Z$ such that 
$r_l\to r$, $s_l \to s$, and $t_l \to t$.
\end{proof}

The following observation will be used 
in Example~\ref{Expl:multiple-convex}.

\begin{Rem} \label{Rem:geod-preserved}
Let $\sig$ be a concial geodesic bicombing on the proper metric space $X$,
and suppose that for some pair of points $x,y$ the consistency condition
$\sig_{x'y'}(\lam) = \sig_{xy}((1-\lam)s + \lam t)$ holds for all
$0 \le s \le t \le 1$ and $\lam \in [0,1]$, where 
$x' := \sig_{xy}(s)$ and $y' := \sig_{xy}(t)$.
Then it is easily seen that $\sig_{xy}$ as well as all 
positively oriented subsegments are unaltered by the 
procedure in the above proof. In other words, the resulting convex 
bicombing $\bar\sig$ satisfies $\bar\sig_{x'y'} = \sig_{x'y'}$ for all
$x',y'$ as above, in particular $\bar\sig_{xy} = \sig_{xy}$.
\end{Rem}


\section{Straight curves}

Whereas the preceding section dealt with the existence of convex bicombings, 
we now turn to the question of uniqueness. First we consider a
property every geodesic from a convex bicombing necessarily
shares. 

Let $X$ be a metric space. We call a curve $\gam \colon [a,b]
\to X$ {\em straight} (or a {\em straight segment}) 
if for every $z \in X$ the function $d_z \circ \gam$ is convex, where
$d_z = d(z,\cdot)$. In particular, for fixed 
$s,t \in [a,b]$, taking $z := \gam(s)$ one gets the inequality
\[
d(\gam(s),\gam((1-\lam)s + \lam t)) \le \lam \,d(\gam(s),\gam(t)) 
\quad \text{for all $\lam\in[0,1]$,}
\]
whereas for $z := \gam(t)$ one obtains
\[
d(\gam((1-\lam)s + \lam t),\gam(t)) \le (1-\lam) \,d(\gam(s),\gam(t))
\quad \text{for all $\lam\in[0,1]$.}
\]
By taking the sum of these two inequalities one sees
that straight curves are geodesics (of constant speed). The terminology 
is further justified by Theorem~\ref{Thm:straight-in-vs} below.
In the proof of this result as well as in Example~\ref{Expl:multiple-convex} 
we use Isbell's injective hull construction~\cite{Isb}, 
which we first recall (see~\cite{L} for a more detailed exposition).

Given a metric space $X$, consider the vector space $\R^X$ of arbitrary 
real valued functions on $X$ and put
\[
\Del(X) := \{ f \in \R^X : 
\text{$f(x) + f(y) \ge d(x,y)$ for all $x,y \in X$}\}.
\]
Call $f \in \Del(X)$ {\em extremal\/} if there is no $g \le f$ in $\Del(X)$ 
distinct from $f$. The set $\E(X)$ of extremal functions is equivalently 
described as
\[
\E(X) = \bigl\{ f \in \R^X : 
\text{$f(x) = \textstyle\sup_{y \in X}(d(x,y) - f(y))$ for all $x \in X$} 
\bigr\},
\]
and it is easily seen that $\E(X)$ is a subset of
\[
\Del_1(X) := \{f \in \Del(X): \text{$f$ is $1$-Lipschitz}\}.
\]
Note that a function $f \in \R^X$ belongs to $\Del_1(X)$ if and only if 
\begin{equation} \label{eq:f-d}
\|f - d_x\|_\infty 
= f(x) \quad \text{for all $x \in X$.}
\end{equation}
The metric $(f,g) \mapsto \|f - g\|_\infty$ on $\Del_1(X)$ is thus finite,
as $\|f - g\|_\infty \le \|f - d_x\|_\infty + \|g - d_x\|_\infty = f(x) + g(x)$ 
for any $x \in X$. The set $\E(X) \sub \Del_1(X)$ is equipped with the induced 
metric, and one has the canonical isometric embedding 
\[
\e \colon X \to \E(X), \quad \e(x) = d_x.
\]
Isbell showed that $(\e,\E(X))$ is indeed
an injective hull of $X$. That is, $\E(X)$ is an injective metric space, and 
every isometric embedding of $X$ into another injective metric space 
factors through $\e$. 

Returning to straight curves, we first observe that this property 
persists when we pass from $X$ to $\E(X)$.

\begin{Lem} \label{Lem:straightXtoEX}
Let $\gam \colon [a,b] \to X$ be a straight curve in some metric
space $X$. Regarding $X$ as a subspace of its injective hull $\E(X)$
we then have that $\gam$ is also straight in $\E(X)$.
\end{Lem}

\begin{proof}
By~\eqref{eq:f-d}, the distance from an element $f\in \E(X)$ 
to a point $x \in X \sub \E(X)$ equals $f(x)$. So
we need to show that the function $f \circ \gam$ is convex. 
Given $x := \gam(s)$, $y := \gam(t)$, and $w := \gam((1-\lam)s + \lam t)$,
where $\lam \in [0,1]$, let $\eps > 0$ and choose (using the extremality
of $f$) a point $v \in X$ such that $f(v) + f(w) \le d(v,w) + \eps$. 
Since $\gam$ is straight in $X$,  
we have
\[
d(v,w) \le (1-\lam)\,d(v,x) + \lam\,d(v,y).
\]
Furthermore, $d(v,x) \le f(v) + f(x)$ and 
$d(v,y) \le f(v) + f(y)$. Combining these inequalities we get
\[
f(v) + f(w) 
\leq f(v) + (1-\lambda) f(x) + \lambda f(y) + \eps.
\]
Since $\eps$ was arbitrary, this gives $f(w) \le 
(1-\lambda) f(x) + \lambda f(y)$, as desired.
\end{proof}

Now let $X = V$ be a normed real vector space. Isbell~\cite{Isb2} and 
Rao~\cite{Rao} (see also~\cite{CiaD}) showed that then the injective 
hull $\E(V)$ has a Banach space structure with respect to which 
the isometric embedding $\e \colon V \to \E(V)$ is linear.
Since $\E(V)$ is injective, collections of balls in $\E(V)$ have 
the binary intersection property, so the Banach space $\E(V)$ is also 
injective in the linear category by~\cite{Nac}. Then a theorem of 
Nachbin, Goodner, and Kelley~\cite{Kel} implies that $\E(V)$ 
is isometrically isomorphic to the space $C(M)$ of continuous functions, 
with the supremum norm, on some extremally disconnected compact 
Hausdorff space $M$. Summarizing, we may thus view $V$ as a linear subspace 
of $C(M)$, where $M$ is such that $\E(V) \cong C(M)$. 
This fact will be used below.

As usual, we call a curve $\gam \colon [a,b] \to V$ in a vector space 
$V$ {\em linear} if it is of the form $t\mapsto p+ty$
for some $p,y \in V$. This is obviously a local property.

\begin{Prop} \label{Prop:cvcs}
Let $M$ be a compact Hausdorff space, and let 
$\gam \colon [a,b] \to O$ be a curve in an open subset $O$ of $C(M)$.
Then $\gam$ is straight (in $O$ with the induced metric) if and only if
it is linear.
\end{Prop}

\begin{proof}
Clearly every linear curve is straight. For the other direction,
we assume that $\gam \colon [0,1] \to O$ is a straight curve from $0$ to $y$,
where $\|y\|_\infty = l$, and the closed $2l$-neighborhood of $\gam([0,1])$
is contained in $O$. We have to show that the two functions
$\gam(\lam)$ and $\lam y$ agree for every $\lam \in (0,1)$.
So fix $\lambda \in (0,1)$ as well as $m \in M$.
For an arbitrary $\eps > 0$, choose an open neighborhood $U$ of $m$ such
that 
\[
|\gam(\lambda)(m') - \gam(\lambda)(m)| < \eps/2, \quad 
|y(m') - y(m)| < \eps/2
\]
for all $m' \in U$.
By Urysohn's lemma there is a nonnegative function $\phi \in C(M)$ with
$\phi|_{M\setminus U} \equiv 0$ and $\|\phi\|_\infty = 2l$. Put 
$z_{\pm} := \gam(\lam) \pm \phi$ and note that 
these are elements of $O$. Since $\gam$ is straight,
\begin{equation} \label{eq:zpm}
2l = \|z_\pm - \gam(\lam)\|_\infty 
\le (1-\lam)\|z_\pm\|_\infty + \lam \|z_\pm - y\|_\infty.
\end{equation}
Now if $f \in C(M)$ is a function satisfying $\|f\|_\infty \le l$ as well as
$f(m') < f(m) + \eps$ for all $m' \in U$, 
then $(f + \phi)|_U < 2l + f(m) + \eps$,
$(f + \phi)|_{M \sm U} \le l$, $-(f + \phi) \le -f \le l$, and 
$l \le 2l + f(m)$, hence
\[
\|f + \phi\|_\infty \le 2l + f(m) + \eps.
\]
Taking $f = \gam(\lam)$ and $f = \gam(\lam) - y$ we get 
$\|z_+\|_\infty \le 2l + \gam(\lam)(m) + \eps$ and 
$\|z_+ - y\|_\infty \le 2l + \gam(\lam)(m) - y(m) + \eps$, respectively.
(Note that $\|f\|_\infty \le l$ since $\gam$ is a geodesic from $0$ to $y$.)
Together with~\eqref{eq:zpm}, this shows that
\[
\gam(\lam)(m) - \lam y(m) + \eps \ge 0.
\]
Similarly, for $f = -\gam(\lam)$ and $f = y - \gam(\lam)$
we obtain $\|z_-\|_\infty \le 2l - \gam(\lam)(m) + \eps$ and 
$\|z_- - y\|_\infty \le 2l + y(m) - \gam(\lam)(m) + \eps$, respectively,
which gives
\[
\lam y(m) - \gam(\lam)(m) + \eps \ge 0.
\]
Letting $\eps \to 0$ we conclude that $\gam(\lam)(m) = \lam y(m)$,
ending the proof as $\lam$ and $m$ were arbitrary.
\end{proof}

Combining these results we obtain the desired characterization
of linear geodesics in normed spaces.

\begin{Thm} \label{Thm:straight-in-vs}
A straight curve in an arbitrary normed space $V$ is linear. Hence the 
bicombing of linear geodesics is the only convex bicombing on $V$.
\end{Thm}

\begin{proof}
As mentioned above, we can regard the normed space $V$ as being a linear
subspace of its injective hull $\E(V)$, and we have $\E(V) \cong C(M)$ for
some compact Hausdorff space $M$. Now straight curves in $V$ are straight in 
$C(M)$ by Lemma~\ref{Lem:straightXtoEX} and linear by 
Proposition~\ref{Prop:cvcs}.
\end{proof}

One may ask whether this holds more generally for linearly convex subsets
of normed spaces. The answer turns out to be negative.

\begin{Expl} \label{Expl:nonlinear-straight}
Consider the Banach space $\ell_\infty([0,1])$ of bounded real valued 
functions on $[0,1]$, equipped with the supremum norm. The subset
\[
C:= \left\{ f\in\ell_\infty([0,1]) : 
\text{$f(0)+f(1) = 1$, $f$ is convex, and $f\in\Delta_1([0,1])$} \right\}
\]
is compact and linearly convex. Hence there is the convex bicombing of linear 
geodesics in $C$. But there are more straight segments: the curve
$\gam \colon [0,1] \to C$, $\gam(t) = d_t$, is non-linear,
and by~\eqref{eq:f-d} we have $t \mapsto \|f - \gam(t)\|_\infty = f(t)$ 
for every $f \in C$, which is a convex function by the definition of $C$.
\end{Expl}

We conclude this section with an example of a compact metric space 
that admits at least two distinct convex geodesic bicombings.

\begin{Expl} \label{Expl:multiple-convex}
Consider two geodesics $\alpha,\beta \colon [0,1] \to \ell_\infty([0,1])$ from
$d_0$ to $d_1$: the linear geodesic $\alpha(s) = (1-s)d_0+sd_1$,
and the Kuratowski embedding $\beta(t) = d_t$ of $[0,1]$.
Let $B \sub \ell_\infty([0,1])$ be the bigon composed of these two 
geodesic segments. By~\eqref{eq:f-d} we have
\begin{equation} \label{eqn:st2st}
\| \alpha(s)-\beta(t) \|_\infty = \alpha(s)(t) = s+t-2st. 
\end{equation}
Since this last term is symmetric in $s$ and $t$, there is an isometric 
involution $\iota \colon B \to B$ that interchanges $\alpha$ and $\beta$.
Furthermore, as $\ell_\infty([0,1])$ is injective, we can embed
the injective hull $\E(B)$ of $B$ into $\ell_\infty([0,1])$ (and
identify it with its image), so that $B \subset \E(B) \subset
\ell_\infty([0,1])$. Now $\E(B)$ is not linearly convex in $\ell_\infty([0,1])$,
but by retracting the linear geodesics with endpoints in $\E(B)$
to $\E(B)$ we obtain a conical geodesic bicombing $\sig$ on $\E(B)$
(compare Lemma~\ref{Lem:retr}). Note that since $\alpha$ was already linear, 
we have $\sig_{d_0 d_1} = \alpha$. Note further that $\E(B)$ is compact, because
$B$ is compact. Theorem~\ref{Thm:intro-1} then yields a convex bicombing
$\bar\sig$ which, by Remark~\ref{Rem:geod-preserved}, still satisfies
$\bar\sig_{d_0 d_1} = \alpha$. Finally, the involution 
$\iota$ extends uniquely to an isometry $I$ of $\E(B)$ 
(see, for instance, Proposition~3.7 in~\cite{L}). Mapping 
$\bar\sig$ by $I$ we get a convex bicombing $\bar\tau$ on $\E(B)$ with
$\bar\tau_{d_0 d_1} = \beta$, distinct from $\bar\sig$.
\end{Expl}


\section{Combinatorial dimension}

The example just described contrasts with Theorem~\ref{Thm:intro-2}, which 
we will prove in this section. First we discuss the structure of injective 
hulls of finite metric spaces and the notion of combinatorial dimension 
in more detail.

Let $X$ be a finite metric space, with $|X| = n \ge 1$, say. 
The set $\Del(X) \sub \R^X \cong \R^n$ is an unbounded
polyhedral domain, determined by the finitely many linear inequalities 
$f(x) + f(y) \ge d(x,y)$ for $x,y \in X$ (in particular $f \ge 0$). 
As $X$ is finite, a function $f \in \Del(X)$ is extremal if and only 
if for every $x \in X$ there exists a point $y \in X$ such that 
$f(x) + f(y) = d(x,y)$. So the injective hull is a polyhedral subcomplex 
of $\partial \Del(X)$ of dimension at most $n/2$. (It is not difficult to see 
that $\E(X)$ consists precisely of the bounded faces of $\partial \Del(X)$.)
For $n \le 5$, the various possible combinatorial types of $\E(X)$ are 
depicted in Sect.~1 of~\cite{Dre} (where $\Del(X)$ and $\E(X)$ are denoted 
$P_X$ and $T_X$, respectively). 
To describe the structure of $\E(X)$ further, one may assign to 
every $f \in \E(X)$ the undirected graph with vertex set $X$ and edge set
\[
A(f) = \bigl\{ \{x,y\} : x,y \in X,\, f(x) + f(y) = d(x,y) \bigr\}.
\]
Note that this graph has no isolated vertices (because $f$ is extremal),
but may be disconnected, and there is a loop $\{x,x\}$ if and only if 
$f(x) = 0$, which occurs if and only if $f = d_x$ (by~\eqref{eq:f-d}). 
Call a set $A$ of unordered pairs of (possibly equal) points in $X$ 
{\em admissible} if there exists an $f \in \E(X)$ with $A(f) = A$,
and denote by $\cA(X)$ the collection of admissible sets.
The family of polyhedral faces of $\E(X)$ is then given by 
$\{P(A)\}_{A \in \cA(X)}$, where 
\[
P(A) = \{f \in \Del(X): A \sub A(f)\},
\]
and where $P(A')$ is a face of $P(A)$ if and only if $A \sub A'$.
We define the {\em rank\/} $\rk(A)$ of $A$ as the dimension of $P(A)$.
This number can be read off as follows. If $f,g$ are two elements of 
$P(A)$, then $f(x) + f(y) = d(x,y) = g(x) + g(y)$ for
$\{x,y\} \in A$, so $f(y) - g(y) = -(f(x) - g(x))$. Thus the difference 
$f-g$ has alternating sign along all edge paths in the graph $(X,A)$. 
It follows that there is either no or exactly one degree of freedom 
for the values of $f \in P(A)$ on every connected component of $(X,A)$, 
depending on whether or not the component contains a cycle of odd length. 
We call such components (viewed as subsets of $X$) {\em odd} or 
{\em even $A$-components}, respectively. 
The rank $\rk(A) = \dim(P(A))$ is then precisely the number of even 
$A$-components of~$X$. (Here we have adopted the notation 
from~\cite{L}, whereas~\cite{Dre} uses 
$\cK_f = \{(x,y) \in X \times X : f(x) + f(y) = d(x,y)\}$ in place 
of $A(f)$.)

Now let again $X$ be a general metric space.
We recall that the {\em combinatorial dimension\/} of $X$, 
introduced by Dress, is the possibly infinite number 
\[
\dim_{\comb}(X) = \sup\{ \dim(\E(Y)): Y \sub X,\, |Y| < \infty \}, 
\]
see Theorem~$9'$ on p.~380 in~\cite{Dre}. This theorem contains in particular
a characterization of the inequality 
$\dim_{\comb}(X) \le n$ in terms of a $2(n+1)$-point inequality, 
which may be rephrased as follows.

\begin{Thm}[Dress] \label{Thm:dress}
Let $X$ be a metric space, and let $n \ge 1$ be an integer.
The inequality $\dim_{\comb}(X) \le n$ holds if and only if 
for every set $Z \sub X$ with $|Z| = 2(n+1)$ and every fixed point free 
involution $i \colon Z \to Z$ there exists a fixed point free bijection
$j \colon Z \to Z$ distinct from $i$ such that
\begin{equation} \label{eq:dress}
\sum_{z \in Z} d(z,i(z)) \le \sum_{z \in Z} d(z,j(z)).
\end{equation}
\end{Thm}

The case $n = 1$ corresponds to the much simpler
fact that $\dim_\comb(X) \le 1$ if and only if $X$ is {\em $0$-hyperbolic\/} 
in the sense of Gromov~\cite{Gro} or {\em tree-like\/} in the terminology 
of~\cite{Dre}, that is, for every quadruple of points $x,x',y,y' \in X$,
\[
d(x,x') + d(y,y') \le \max\{d(x,y)+d(x',y'),d(x,y') + d(x',y)\}.
\]
Theorem~\ref{Thm:dress} follows from more general considerations 
in Sect.~(5.3) of~\cite{Dre}. For convenience, we provide a streamlined and 
somewhat simplified argument in an appendix. However, this result will not 
be used in the present paper.

By the results of~\cite{L}, every proper metric space with integer 
valued metric that is discretely geodesic and $\del$-hyperbolic
has finite combinatorial dimension.
By contrast, the unit circle $S$ in the Euclidean 
plane with either the induced (chordal) or the induced inner metric satisfies
$\dim_{\comb}(S) = \infty$, as is seen by looking at the constant extremal 
function $f = \diam(S)/2$, restricted to the vertices of a regular $2n$-gon. 
Similarly, the metric bigon $B$ constructed in 
Example~\ref{Expl:multiple-convex} has $\dim_\comb(B) = \infty$: 
consider the function $f$ defined by 
$f(\alpha(s)) = f(\beta(1-s)) = \|\alpha(s) - \beta(1-s)\|_\infty/2$ for 
$s \in [0,1]$. 
Among the finite dimensional normed 
spaces, only those with a polyhedral norm have finite combinatorial 
dimension, equal to the number of pairs of opposite
facets of the unit ball. 

The following proposition is the key observation for the proof of 
Theorem~\ref{Thm:intro-2}. We denote by $B(x,r)$ the closed ball of 
radius $r$ at $x$.

\begin{Prop} \label{Prop:key-ineq}
Let $X$ be a metric space of finite combinatorial dimension. Then for
every pair of points $x_0,y_0 \in X$ there exists a $\del > 0$ such
that
\[
d(x_0,y_0) + d(x,y) \le d(x,y_0) + d(x_0,y)
\] 
for all pairs of points $x \in B(x_0,\del)$ and $y \in B(y_0,\del)$.
\end{Prop}

\begin{proof}
When $x_0 = y_0$, the triangle inequality in $X$ gives the result.
Now assume that $x_0 \ne y_0$. Denote by $\cF$ the collection of all
real valued functions with finite support $\spt(f) \sub X$ such that
$f \in \E(\spt(f))$, $x_0,y_0 \in \spt(f)$, and $\{x_0,y_0\} \in A(f)$,
that is,
\[
f(x_0) + f(y_0) = d(x_0,y_0).
\] 
Since $\dim_\comb(X) < \infty$, there exist an integer $n$ 
and a function $f \in \cF$ such that $\rk(A(g)) \le n = \rk(A(f))$ for 
all $g \in \cF$.
Since $x_0 \ne y_0$, we have $n \ge 1$, thus $f > 0$.
By restricting $f$ to a smaller set if necessary, we can arrange that
$A(f)$ is the collection of $n$ disjoint pairs 
$\{x_0,y_0\},\{x_1,y_1\},\dots,\{x_{n-1},y_{n-1}\}$.
There exists a $\del > 0$ such that for all $v \in \{x_0,y_0\}$ and
$w \in \{x_1,y_1,\dots,x_{n-1},y_{n-1}\}$ we have $d(v,w) > \del$, 
$f(v) \ge \del$, and
\begin{equation} \label{eq:fvfw}
f(v) + f(w) \ge d(v,w) + 2\del.
\end{equation}
Note that $d(x_0,y_0) = f(x_0) + f(y_0) \ge 2\del$. Let $x \in
B(x_0,\del)$ and $y \in B(y_0,\del)$. If $x = x_0$ or $y = y_0$ or $x
= y$, the desired inequality holds. So assume that $x_0 \ne x \ne y
\ne y_0$. Then $x,y,x_0,y_0,\dots,x_{n-1},y_{n-1}$ are pairwise
distinct. Put $a := d(x,y_0) - f(y_0)$ and $b := d(x_0,y) -
f(x_0)$. We have
\begin{align} \label{eq:alpha-beta}
\begin{split}
a &\ge d(x_0,y_0) - d(x,x_0) - f(y_0) = f(x_0) - d(x,x_0) 
\ge f(x_0) - \del, \\ 
b &\ge d(x_0,y_0) - d(y,y_0) - f(x_0) = f(y_0) - d(y,y_0) 
\ge f(y_0) - \del.
\end{split}
\end{align} 
Since $f(x_0),f(y_0) \ge \del$, this gives in particular 
$a,b \ge 0$ and hence
\begin{align} \label{eq:fv}
\begin{split}
a + f(x_0) &\ge \del \ge d(x,x_0), \\
b + f(y_0) &\ge \del \ge d(y,y_0).
\end{split}
\end{align}
Furthermore, for every $w \in \{x_1,y_1,\dots,x_{n-1},y_{n-1}\}$,
combining~\eqref{eq:alpha-beta} and~\eqref{eq:fvfw} we obtain
\begin{align} \label{eq:fw}
\begin{split}
a + f(w) &\ge f(x_0) + f(w) - \del \ge d(x_0,w) + \del \ge d(x,w), \\ 
b + f(w) &\ge f(y_0) + f(w) - \del \ge d(y_0,w) + \del \ge d(y,w).
\end{split}
\end{align}
Now, in the case that $a + b < d(x,y)$, we could define a function 
$g$ with finite support by putting $g(w) := f(w)$ for 
$w \in \{x_0,y_0,\dots,x_{n-1},y_{n-1}\}$ and by choosing $g(x) > a$
and $g(y) > b$ such that $g(x) + g(y) = d(x,y)$. In view 
of~\eqref{eq:fv} and~\eqref{eq:fw}, this function would satisfy
\[
A(g) = \bigl\{\{x_0,y_0\},\dots,\{x_{n-1},y_{n-1}\},\{x,y\}\bigl\},
\]
so that $g \in \cF$ and $\rk(A(g)) = n+1$, in contradiction to the 
maximality of $n$. We conclude that $d(x,y) \le a + b = d(x,y_0) +
d(x_0,y) - d(x_0,y_0)$.
\end{proof}

Resuming the discussion of straight curves,
we can now prove the following.

\begin{Prop} \label{Prop:unique-straight}
Suppose that $X$ is a metric space of finite combinatorial dimension,
and $\alpha,\beta \colon [0,1] \to X$ are two straight curves. Then the
function $s \mapsto d(\alpha(s),\beta(s))$ is convex on $[0,1]$. In
particular, every pair of points in $X$ is joined by at most one straight
segment, up to reparametrization.
\end{Prop}

\begin{proof}
For $s,t \in [0,1]$, put $h(s,t) := d(\alpha(s),\beta(t))$. Fix $s_0 \in
(0,1)$. Then it follows from Proposition~\ref{Prop:key-ineq} that there
exists an $\eps > 0$ such that $[s_0 - \eps,s_0 + \eps] \sub [0,1]$
and, for all $s,t \in [s_0 - \eps,s_0 + \eps]$,
\begin{equation} \label{eq:h}
h(s_0,s_0) + h(s,t) \le h(s,s_0) + h(s_0,t).
\end{equation}
Now suppose that $s_0 - \eps \le s < s_0 < t \le s_0 + \eps$, and let
$\lam \in (0,1)$ be such that $s_0 = (1-\lam) s + \lam t$. Since
$h(s,\cdot)$ and $h(\cdot,t)$ are convex functions on $[0,1]$, we have
\begin{align} \label{eq:hh}
\begin{split}
h(s,s_0) &\le (1-\lam)\,h(s,s) + \lam\,h(s,t),\\ 
h(s_0,t) &\le (1-\lam)\,h(s,t) + \lam\,h(t,t).
\end{split}
\end{align}
Combining~\eqref{eq:h} and~\eqref{eq:hh} we conclude that
\[
h(s_0,s_0) \le (1-\lam)\,h(s,s) + \lam\,h(t,t).
\]
Note that this holds whenever $s_0 - \eps \le s < s_0 < t \le s_0 +
\eps$ and $s_0 = (1-\lam) s + \lam t$, where $\eps > 0$ depends on
$s_0$. Since $s \mapsto h(s,s) = d(\alpha(s),\beta(s))$ is
continuous on $[0,1]$, it follows easily that this function is convex.
\end{proof}

Theorem~\ref{Thm:intro-2} is an immediate corollary
of Proposition~\ref{Prop:unique-straight}. Note that a convex 
(or conical) geodesic bicombing $\sig$ can be restricted 
to any ball $B = B(z,r)$ because $B$ is {\em $\sig$-convex}:
$x,y \in B$ implies $\sig_{xy}([0,1]) \subset B$.


\section{Boundary at infinity}

We now consider a complete metric space $X$ with a convex and consistent 
(equivalently, conical and consistent) geodesic bicombing.
Note that completeness is no restriction, as any conical bicombing may 
be extended to the completion of the underlying space. 
We define the geometric boundary and the 
closure of~$X$ by means of geodesic rays that are consistent with 
the given bicombing, and we equip the closure with a simple explicit metric. 
(Some general references for the analogous constructions in the case 
of $\CAT(0)$ spaces or Gromov hyperbolic spaces are~\cite{Bal,BriH,BuyS,Gro}.) 
Then we prove Theorem~\ref{Thm:intro-4}.

By a {\em ray\/} in $X$ we mean an isometric embedding of $\R_+ := [0,\infty)$.
Two rays $\xi,\eta$ in $X$ are {\em asymptotic\/} if the function
$t \mapsto d(\xi(t),\eta(t))$ is bounded or, equivalently, the Hausdorff 
distance between the images of $\xi$ and $\eta$ is finite. In the presence 
of a convex and consistent bicombing $\sig$ on $X$ we call a ray 
$\xi \colon \R_+ \to X$ a {\em $\sig$-ray\/} if
\[
\xi((1-\lam)s + \lam t) = \sigma_{xy}(\lam)
\]
whenever $0 \le s \le t$, $x := \xi(s)$, $y := \xi(t)$, and 
$\lam \in [0,1]$. It follows that for any two $\sig$-rays $\xi,\eta$
the map $t \mapsto d(\xi(a + \alpha t),\eta(b + \beta t))$ 
is convex on $\R_+$ for all $a,\alpha,b,\beta \ge 0$. 

The {\em geometric boundary\/} $\di_\sig X$ is the set of equivalence 
classes of mutually asymptotic $\sig$-rays in $X$, and we write
\[
\olX_\sig := X \cup \di_\sig X.
\]
For a unified treatment of the two parts of $\olX_\sig$ it is convenient 
to associate with every pair $(o,x) \in X \times X$ the eventually constant
curve $\rho_{ox} \colon \R_+ \to X$, 
\begin{equation} \label{eq:rho-ox}
\rho_{ox}(t) := 
\begin{cases}
\sig_{ox}(t/d(o,x)) & \text{if $0 \le t < d(o,x)$,} \\
x                  & \text{if $t \ge d(o,x)$.}
\end{cases}
\end{equation}
As a preliminary remark we note that for any basepoint $o \in X$ and 
$r \in \R_+$, the radial retraction $\phi_r \colon X \to B(o,r)$ defined by
$\phi_r(x) := \rho_{ox}(r)$
satisfies 
\begin{equation} \label{eq:phi-r}
d(\phi_r(x),\phi_r(y)) \le \frac{2r}{d(o,x)} \,d(x,y)
\end{equation}
whenever $d(o,x) \ge d(o,y)$ and $d(o,x) > r$.
To see this, let $s := r/d(o,x)$ and $y' := \sig_{oy}(s)$.
We have $\phi_r(x) = \sig_{ox}(s)$, and since $\sig$ is conical,
$d(\phi_r(x),y') \le s\,d(x,y)$. Furthermore, 
$d(y',\phi_r(y)) = d(o,\phi_r(y)) - d(o,y') \le d(o,\phi_r(x)) - d(o,y')
\le d(\phi_r(x),y')$ 
and so $d(\phi_r(x),\phi_r(y)) \le d(\phi_r(x),y') + d(y',\phi_r(y)) 
\le 2s\,d(x,y)$.
In particular, $\phi_r$ is $2$-Lipschitz. The constant~2 is optimal, 
as one can see by looking at the space $\ell^2_\infty$ with $0$ 
as the basepoint, the map $\phi_1$, and the points $(1,1), (1+\eps,1-\eps)$. 

In order to prove that for every pair $(o,\bar x) \in X \times \di_\sig X$ 
there is a $\sig$-ray issuing from $o$ and representing the class $\bar x$,
we shall need the following estimate (compare Lemma~II.8.3 in~\cite{BriH} 
for the case of $\CAT(0)$ spaces).

\begin{Lem} \label{Lem:t-T-estimate}
Let $X$ be a metric space with a convex and consistent geodesic bicombing 
$\sig$, and let $o,p \in X$. Then for any $\sig$-ray $\xi$ with $\xi(0) = p$ 
we have
\[
d(\rho_{ox}(t),\rho_{oy}(t)) \le \frac{2t\,d(o,p)}{T - d(o,p)}
\]
whenever $T > 2d(o,p)$, $x,y \in \xi([T,\infty))$, and 
$0 \le t \le T - 2\,d(o,p)$. 
\end{Lem}

\begin{proof} 
Assume $d(o,x) \le d(o,y)$ and let $s := d(p,x)/d(p,y)$ and 
$y' := \sig_{oy}(s)$. 
Since $\xi$ is a $\sig$-ray, we have $x = \sig_{py}(s)$.
As $\sig$ is conical, 
\begin{equation} \label{eq:xy-op}
d(x,y') \le (1-s)\,d(p,o) \le d(o,p).
\end{equation}
Now $d(o,x) \ge d(p,x) - d(o,p) \ge T - d(o,p)$ and so
$d(o,y') \ge d(o,x) - d(x,y') \ge T - 2\,d(o,p)$. 
Hence, for $0 \le t \le T - 2\,d(o,p)$, we have
$\phi_t(x) = \rho_{ox}(t)$ and $\phi_t(y')=\rho_{oy}(t)$,
and~\eqref{eq:phi-r} gives the result.
\end{proof}

The next result now follows by a standard procedure.

\begin{Prop} \label{Prop:rho-o-x}
Let $X$ be a complete metric space with a convex and consistent
geodesic bicombing $\sig$. Then for every pair 
$(o,\bar x) \in X \times \di_\sig X$ 
there is a unique $\sig$-ray $\rho_{o\bar x}$ with $\rho_{o\bar x}(0) = o$ that 
represents the class $\bar x$. Furthermore, if $r \ge 0$ and 
$x = \rho_{o\bar x}(r)$, then $\rho_{x\bar x}(t) = \rho_{o\bar x}(r+t)$ for all 
$t \in \R_+$. 
\end{Prop}

\begin{proof}
Let $\xi$ be a $\sig$-ray in the class $\bar x$, and let 
$p := \xi(0)$. For $n = 1,2,\dots$, put $\rho_n := \rho_{o\xi(n)}$.
It follows from the preceding lemma that for every fixed $t \ge 0$ the 
sequence $\rho_n(t)$ is Cauchy. In the limit one obtains a $\sig$-ray 
$\rho_{o\bar x}$ issuing from~$o$. As in~\eqref{eq:xy-op}, we have
$d(\xi(t),\rho_n[t\,d(o,\xi(n))/n]) \le d(o,p)$ for $0 \le t \le n$, 
hence $d(\xi(t),\rho_{o\bar x}(t)) \le d(o,p)$ for all $t \ge 0$. 
In particular, $\rho_{o\bar x}$ is asymptotic to~$\xi$.
Finally, if $\rho'$ is another $\sig$-ray issuing from $o$ and asymptotic
to~$\xi$, then $t \mapsto d(\rho_{o\xi}(t),\rho'(t))$ is a non-negative, 
bounded, convex function on $\R_+$ that vanishes at $0$, 
so $\rho' = \rho_{o\xi}$. From this uniqueness property, the last
assertion of the proposition is clear.
\end{proof}

A natural topology on $\olX_\sig = X \cup \di_\sig X$ may be described 
in different ways. First we fix a basepoint $o \in X$ and consider the
set
\[
\olR_{\sig,o} := \{\rho_{o\bar x} : \bar x \in \olX_\sig\}
= \{\rho_{ox} : x \in X\} \cup \{\rho_{o\bar x} : \bar x \in \di_\sig X\}
\]
of {\em generalized $\sig$-rays based at $o$}, given by~\eqref{eq:rho-ox}
and Proposition~\ref{Prop:rho-o-x}. 
We equip $\olR_{\sig,o}$ with the topology of 
uniform convergence on compact subsets of $\R_+$. Clearly $\olR_{\sig,o}$ is
compact if and only if $X$ is proper, as a consequence of the 
Arzel\`a--Ascoli theorem. The topology of $\olR_{\sig,o}$ agrees, 
under canonical identification, with the {\em cone topology\/} 
on $\olX_\sig$, a basis of which is given by the sets
\begin{equation} \label{eq:cone-top}
U_o(\bar x,t,\eps) 
:= \{ \bar y \in \olX_\sig : d(\rho_{o\bar x}(t),\rho_{o\bar y}(t)) < \eps \}
\end{equation}
for $\bar x \in \olX_\sig$ and $t,\eps > 0$. 
Note that since $\phi_r$ is $2$-Lipschitz, we have  
\begin{equation} \label{eq:rho-r-t}
d(\rho_{o\bar x}(r),\rho_{o\bar y}(r)) \le 2\,d(\rho_{o\bar x}(t),\rho_{o\bar y}(t))
\quad \text{for all $r \in [0,t]$.}
\end{equation}
Note also that $U_o(x,t,\eps)$ is just 
the open ball $U(x,\eps)$ in case $t \ge d(o,x) + \eps$.
It follows readily from the following lemma that this topology 
on $\olX_\sig$ is independent of the choice of basepoint $o$.

\begin{Lem} \label{Lem:basept-inv}
Given $o \in X$, $\bar x \in \di_\sig X$, $\eps,t > 0$, and $p \in X$, there 
exists $T > 0$ such that $U_p(\bar x,\eps/4,T) \sub U_o(\bar x,\eps,t)$.  
\end{Lem}

\begin{proof}
It follows from Lemma~\ref{Lem:t-T-estimate} and the construction of the 
ray $\rho_{o\bar x}$ in Proposition~\ref{Prop:rho-o-x}
that if $T$ is chosen sufficiently large, depending on $d(o,p)$ and 
$\eps,t$, and if $x := \rho_{p\bar x}(T)$, then 
\[
d(\rho_{ox}(t),\rho_{o\bar x}(t)) \le \eps/4.
\]
Likewise, for any point $\bar y \in \olX_\sig$, if $y := \rho_{p\bar y}(T)$
and $d(p,y)$ is large enough, then
\[
d(\rho_{oy}(t),\rho_{o\bar y}(t)) \le \eps/4.
\]
Now if $\bar y \in U_p(\bar x,T,\eps/4)$, that is, $d(x,y) < \eps/4$, then
\[
d(\rho_{ox}(t),\rho_{oy}(t)) \le 2\,d(x,y) < \eps/2
\]
by~\eqref{eq:phi-r}, provided $d(o,x),d(o,y) > t$. We conclude that
$d(\rho_{o\bar x}(t),\rho_{o\bar y}(t)) < \eps$ and thus 
$\bar y \in U_o(\bar x,t,\eps)$ for sufficiently large $T$.
\end{proof}

Next we equip $\olX_\sig$ with the metric defined by
\begin{equation}\label{eq:D-o}
D_o(\bar x,\bar y) 
:= \int_0^\infty d(\rho_{o\bar x}(s),\rho_{o\bar y}(s)) \, e^{-s} \, ds
\end{equation}
(compare Sect.~8.3.B in~\cite{Gro}).
We have $D_o(o,\bar x) \le 1$, with equality 
if and only if $\bar x \in \di_\sig X$. 
For $d(\rho_{o\bar x}(t),\rho_{o\bar y}(t)) = a$, 
\eqref{eq:rho-r-t}~yields
\[
\int_t^\infty \frac{a}{2}\, e^{-s} \,ds
\le D_o(\bar x,\bar y) 
\le \int_0^t 2a\, e^{-s} \,ds + \int_t^\infty 2s\, e^{-s} \,ds,
\]
and it follows easily from these estimates that the metric~\eqref{eq:D-o} 
induces the cone topology. 
Observe also that if $d(o,x) = R$ and $y = \phi_r(x)$ for some 
$0 \le r \le R$, then 
\[
D_o(x,y) 
= \int_r^R (s-r)\,e^{-s}\,ds + \int_R^\infty (R-r)\,e^{-s}\,ds
= e^{-r} - e^{-R}.
\]
In particular, for any $\sig$-ray $\xi$ issuing from $o$,
the curve $\lam \mapsto \xi(-\log(1-\lam))$, $\lam \in [0,1)$, is a unit speed 
geodesic with respect to $D_o$.
Accordingly, for $\lam \in [0,1]$, we define the radial retraction 
$\psi_\lam \colon \olX_\sig \to 
B_{D_o}(o,\lam) = \{\bar y: D_o(o,\bar y) \le \lam\}$ such that $\psi_1 = \id$ 
and $\psi_\lam(\bar x) := \rho_{o\bar x}(t)$ for $\lam < 1$ and 
$t := -\log(1-\lam)$. In this latter case, the generalized ray $\rho_{ox}$ 
with endpoint $x := \psi_\lam(\bar x)$ agrees with $\phi_t \circ \rho_{o\bar x}$, 
and since $\phi_t$ is $2$-Lipschitz we obtain
\begin{equation} \label{eq:psi-lam-lam}
D_o(\psi_\lam(\bar x),\psi_\lam(\bar y)) 
= \int_0^\infty d(\phi_t(\rho_{o\bar x}(s)),\phi_t(\rho_{o\bar y}(s)))\,e^{-s}\,ds 
\leq 2\,D_o(\bar x,\bar y)
\end{equation}
for all $\bar x,\bar y \in \olX_\sig$. Thus $\psi_\lam$ is $2$-Lipschitz with 
respect to $D_o$. Finally, we note that if $\bar x \in \olX_\sig$ and 
$\lam,\mu \in [0,1]$,
then clearly 
\begin{equation} \label{eq:psi-lam-mu}
D_o(\psi_\lam(\bar x),\psi_\mu(\bar x)) \le |\lam - \mu|,
\end{equation}
with equality when $D_o(o,\bar x) \geq \max\{\lam,\mu\}$.
Now we turn to the result stated in the introduction.

\begin{proof}[Proof of Theorem~\ref{Thm:intro-4}]
The map $(\bar x,\lam) \mapsto \psi_{1-\lam}(\bar x)$ is continuous 
on $\olX_\sig \times [0,1]$ by~\eqref{eq:psi-lam-lam} 
and~\eqref{eq:psi-lam-mu} and contracts $\olX_\sig$ to $o$,
so $\olX_\sig$ is contractible.

To show that $\olX_\sig$ is locally contractible, we prove 
that in fact each of the sets $U_o(\bar x,t,\eps)$ for 
$\bar x \in \olX_\sig$ and $t,\eps > 0$ (see~\eqref{eq:cone-top}) 
is contractible. The same map as above, 
but restricted to $U_o(\bar x,t,\eps) \times [0,e^{-t}]$, contracts 
$U_o(\bar x,t,\eps)$ to the subset $U(\rho_{o\bar x}(t),\eps) \cap B(o,t)$, 
which as an intersection of balls is $\sig$-convex and hence itself 
contractible.

To prove that $\di_\sig X$ is a $Z$-set in $\olX_\sig$, let an open 
set $U$ in $\olX_\sig$ be given, and assume that $U \ne \es,\olX_\sig$. 
We want to find a homotopy $h \colon U \times [0,1] \to U$
from the identity on $U$ to a map into $U \setminus \di_\sig X$ such that 
the restriction of $h$ to $(U \setminus \di_\sig X) \times [0,1]$ takes 
values in $U \setminus \di_\sig X$. To this end, we note that the function
$\mu \colon U \to \R$ defined by 
\[
\mu(\bar x) := D_o(o,\bar x) - \frac12 \inf \{ D_o(\bar x,\bar y) : 
\bar y \in \olX_\sig \setminus U \}
\]
is Lipschitz continuous with respect to $D_o$ and satisfies 
$\mu(\bar x) < D_o(o,\bar x)$ for all $\bar x \in U$ because $U$ is open.
It is then easy to see that 
$h(\bar x,\lam) := \psi_{\max\{1-\lam,\mu(\bar x)\}}(\bar x)$ serves the purpose.

It remains to show that $\olX_\sig$ is an AR (absolute retract).
In case $X$ has finite topological dimension, we have
\[
\dim(\olX_\sig) = \dim(X) < \infty.
\]
Indeed, for every $\eps \in (0,1)$, any open covering of the $D_o$-ball 
$B_{D_o}(o,1-\eps)$ with mesh $\le \eps$ gives rise, via $\psi_{1-\eps}^{\,-1}$,
to an open covering of $\olX_\sig$ with mesh $\le 3\eps$ and the same 
multiplicity (see, for instance, \cite{BuyS} for the definitions).
It is then a standard result that contractible and locally contractible 
metrizable spaces of finite dimension are absolute retracts (see~\cite{Dug}). 
However, finite dimension is not needed.
Every metric space $X$ with a conical geodesic bicombing is
{\em strictly equiconnected}, as defined in~\cite{Him}, so $X$ is an AR by
Theorem~4 in that paper (see also~\cite{Dug2}). 
Now, by Corollary~6.6.7 in~\cite{Sak} (a result attributed to O.~Hanner 
and S.~Lefschetz), $\olX_\sig$ is an ANR (absolute neighborhood retract) 
as it contains $X$ as a homotopy dense subset; alternatively, one can give 
a short direct argument along the lines of Lemma~1.4 in~\cite{Tor}.
Thus $\olX_\sig$ is a contractible ANR or, equivalently, an AR (see 
Corollary~6.2.9 in~\cite{Sak}). 
\end{proof}

As a concluding remark we note that, presumably, the spaces $\di_\sig X$ 
and $\olX_\sig$ depend on the choice of $\sig$. 
However, if $X$ is a complete (not necessarily proper) Gromov hyperbolic 
space with a convex and consistent bicombing $\sig$, then 
$\olX_\sig$ is independent of $\sig$ and furthermore $\di_\sig X$ agrees with 
the boundary $\di_\infty X$ defined in terms of sequences converging to
infinity. Indeed, a modification of Lemma~\ref{Lem:t-T-estimate} and 
Proposition~\ref{Prop:rho-o-x} first shows that for every ray $\xi$ in $X$ 
there is a $\sig$-ray asymptotic to $\xi$ and issuing from the 
same point. As in the proof of Lemma~\ref{Lem:basept-inv} one can then 
conclude that distinct bicombings yield homeomorphic boundaries.


\section*{Appendix: Proof of Theorem~\ref{Thm:dress}}

We use the notation from Sect.~4.

First we show the ``if'' direction.
Let $Y \sub X$ be a finite set. If $|Y| \le 2n + 1$, the dimension of 
$\E(Y)$ is at most $n$. Now suppose
that $|Y| \ge 2n + 2$, $f \in \E(Y)$, and $Z \sub Y$ is a set with 
$|Z| = 2n + 2$ and with a fixed point free involution $i \colon Z \to Z$ 
such that $\{z,i(z)\} \in A(f)$ for every $z \in Z$.
By assumption, there exists a fixed point free bijection $j \colon Z \to Z$ 
such that $j \ne i$ and
\[
\sum_{z \in Z} d(z,i(z)) \le \sum_{z \in Z} d(z,j(z)).
\]
Since $f(z) + f(i(z)) = d(z,i(z))$ and 
$d(z,j(z)) \le f(z) + f(j(z))$, this gives
\[
\sum_{z \in Z} f(z) + f(i(z)) \le \sum_{z \in Z} f(z) + f(j(z)).
\]
However, since both $i$ and $j$ are bijections, these two sums  
agree, so each of the inequalities $d(z,j(z)) \le f(z) + f(j(z))$ must in fact
be an equality, that is, $\{z,j(z)\} \in A(f)$. There is at least one 
$z \in Z$ such that $j(z) \ne i(z)$, so the graph with vertex
set $Z$ and edge set $\bigcup_{z \in Z}\{\{z,i(z)\},\{z,j(z)\}\}$
has at most $n$ connected components.
As this holds for every $Z$ and $i$ as above,
we conclude that also the graph $(Y,A(f))$ has no more than
$n$ components. Since $f \in \E(Y)$ was arbitrary, this shows that the 
dimension of $\E(Y)$ is less than or equal to~$n$.

Now we prove the other implication. Suppose that $Z \sub X$ is a set with 
$|Z| = 2n+2$, and $i \colon Z \to Z$ is a fixed point free involution.
Let $Z_2$ denote the set of all subsets of cardinality two of $Z$.
The involution $i$ selects a subset $Z_i := \{\{z,i(z)\} : z \in Z\}$
of $Z_2$ of $n+1$ disjoint pairs.
For every function $h \colon Z \to \R$ we consider the set $W(h)$ of all
$w \colon Z_2 \to \R$ such that $w \le 0$ on~$Z_i$, $w \ge 0$ 
on~$Z_2 \sm Z_i$, and 
\[
\sum_{z' \in Z \sm \{z\}} w(\{z,z'\}) = h(z)
\]
for all $z \in Z$. First we observe that if $z \in Z$, and if $\{x,y\} \in Z_i$
is chosen such that $z \not\in \{x,y\}$, then the function
$w_{z,\{x,y\}} := (\del_{\{x,z\}} + \del_{\{y,z\}} - \del_{\{x,y\}})/2$
on $Z_2$ belongs to $W(\del_z)$, and $w_{i(z),\{x,y\}} - \del_{\{z,i(z)\}}$ 
is in $W(-\del_z)$. It follows that $W(h)$ is non-empty 
for every $h \colon Z \to \R$. Put
\[
\mu(h) := \sup\{ S(w) : w \in W(h) \}, \quad
S(w) := \sum_{\{x,y\} \in Z_2} w(\{x,y\}) d(x,y);
\]
note that $\mu(h) > -\infty$ but possibly $\mu(h) = \infty$. 
For $w_{z,\{x,y\}}$ as above we have 
$S(w_{z,\{x,y\}}) \ge 0$ by the triangle inequality, thus
$\mu(\del_z) \ge 0$ for all $z \in Z$. Furthermore, 
since $0 \in W(0)$, also $\mu(0) \ge 0$. In fact, if $w \in W(0)$, 
then $\lam w \in W(0)$ for all $\lam \ge 0$, so we have 
either $\mu(0) = 0$ or $\mu(0) = \infty$. 

In case $\mu(0) = \infty$, choose $w \in W(0)$ with $S(w) > 0$ such that
no $w' \in W(0)$ with $S(w') > 0$ has strictly smaller support.
It follows that every nonzero $v \in W(0)$ with support $\spt(v) \sub \spt(w)$
satisfies $S(v) > 0$, for if $\lam > 0$ is the maximal number
with the property that $|\lam v| \le |w|$, then $v' := w - \lam v$ belongs 
to $W(0)$ and $\spt(v')$ is a strict subset of $\spt(w)$,
so $S(v') \le 0$ and $\lam S(v) = S(w) - S(v') > 0$. 
(In fact $\lam v = w$, because $v' \ne 0$ would likewise imply $S(v') > 0$.)
Since for fixed $z$ the sum of the weights $w(\{z,z'\})$ is zero, 
it is not difficult to see that there exist pairwise distinct points 
$z_0,z_1,\dots,z_l$, with $l \ge 1$, such that 
$w(\{z_k,i(z_k)\}) < 0 < w(\{i(z_k),z_{k+1}\})$ 
for $k = 0,\dots,l$, where $z_{l+1} := z_0$. Then the function
\[
v := \sum_{k = 0}^l - \del_{\{z_k,i(z_k)\}} + \del_{\{i(z_k),z_{k+1}\}}
\]
on $Z_2$ belongs to $W(0)$, and $\spt(v) \sub \spt(w)$. Hence $S(v) > 0$. 
This means that 
\[
\sum_{k = 0}^l d(z_k,i(z_k)) < \sum_{k = 0}^l d(i(z_k),z_{k+1}).
\]
The points $i(z_0),\dots,i(z_l)$ are distinct, so there is a well-defined map
$j \colon Z \to Z$ such that $j(i(z_k)) = z_{k+1}$ for $k = 0,\dots,l$ and
$j(z) = i(z)$ otherwise. Note that $j$ is fixed point free since 
$\{i(z_k),z_{k+1}\} \in Z_2$ and $i(z) \ne z$ for all $z \in Z$.
Furthermore, $j$ is injective because $i$ is injective,
$z_0,\dots,z_l$ are distinct, and $j(z) = i(z) = z_k$ would imply
$j(z) = j(i(z_k)) = z_{k+1} \ne z_k$.
Now it is clear that~\eqref{eq:dress} holds, even with strict inequality. 
Note that this part of the proof does not use the assumption
$\dim_\comb(X) \le n$. 

It remains to consider the case $\mu(0) = 0$. Let $h,h' \in \R^Z$. For all
$w \in W(h)$ and $w' \in W(h')$ we have 
$w + w' \in W(h + h')$, therefore
\[
\mu(h) + \mu(h') \le \mu(h + h');
\]
in particular $\mu(h) \le \mu(0) - \mu(-h) = -\mu(-h) < \infty$. Put 
\[
\nu(h) := \frac{\mu(h) - \mu(-h)}{2}.
\] 
We have $\mu(h) - \mu(-h') \ge \mu(h+h')$ and 
$\mu(h') - \mu(-h) \ge \mu(h+h')$, so
\[
\nu(h) + \nu(h') \ge \mu(h+h').
\]
We have already observed that $\mu(\del_z) \ge 0$ for all $z \in Z$, 
hence $-\mu(-\del_z) \ge 0$ and $\nu(\del_z) \ge 0$.
If $\{x,y\} \in Z_2 \sm Z_i$, then 
$\del_{\{x,y\}} \in W(\del_x + \del_y)$, 
thus
\[
\nu(\del_x) + \nu(\del_y) \ge \mu(\del_x + \del_y) \ge d(x,y).
\]
Similarly, if $\{x,y\} \in Z_i$, then $-\del_{\{x,y\}} \in W(-\del_x-\del_y)$, 
hence
\[
\nu(-\del_x) + \nu(-\del_y) \ge \mu(-\del_x - \del_y) 
\ge -d(x,y)
\]
and so $\nu(\del_x) + \nu(\del_y) = -\nu(-\del_x) - \nu(-\del_y)
\le d(x,y)$. Now it is clear that there exists a function 
$f \in \R^Z$ such that $f(z) \ge \nu(\del_z)$ for all $z \in Z$ and 
$f(x) + f(y) = d(x,y)$ for all $\{x,y\} \in Z_i$.
Hence, $f \in \Del(Z)$ and $Z_i \sub A(f)$, so in fact $f \in \E(Z)$.
Since $\dim_\comb(X) \le n$ by assumption, we have $\rk(A(f)) \le n$, 
and there is no loss of generality in assuming that 
there is no $g \in \E(Z)$ with $Z_i \sub A(g)$ and $\rk(A(g)) > \rk(A(f))$.
Since $Z$ has at most $n$ even $A(f)$-components, some such 
component $Z'$ contains more than one of the $n+1$ edges 
in $Z_i$. Note that $i$ maps $Z'$ onto $Z'$. 
We claim that for every $z' \in Z'$ there is a 
$z'' \in Z' \setminus \{z',i(z')\}$ such that $\{z',z''\} \in A(f)$.
If, to the contrary, there were a $z' \in Z'$ with no such $z''$, 
we could choose $g \in \R^Z$ such that $f(z') > g(z') > d(z',z) - f(z)$
and $g(z) = f(z)$ for all $z \in Z \setminus \{z',i(z')\}$, and
$g(i(z')) = d(z',i(z')) - g(z')$. This function would satisfy 
$\rk(A(g)) > \rk(A(f))$ and $Z_i \sub A(g)$, in contradiction to the 
assumption on $f$. 
It follows easily that there exist $z_0,z_1,\dots,z_l \in Z'$, with $l \ge 1$, 
such that $\{z_0,i(z_0)\},\dots,\{z_l,i(z_l)\}$ are pairwise disjoint
and $\{i(z_k),z_{k+1}\} \in A(f)$ for $k = 0,\dots,l$, where $z_{l+1} = z_0$.
Now 
\[
\sum_{k=0}^l d(z_k,i(z_k)) = \sum_{k=0}^l f(z_k) + f(i(z_k))
= \sum_{k=0}^l d(i(z_k),z_{k+1}).
\]
Putting $j(i(z_k)) = z_{k+1}$ and $j(z_{k+1}) = i(z_k)$ for $k = 0,\dots,l$ and  
$j(z) = i(z)$ otherwise, we get a fixed point free involution 
$j \colon Z \to Z$ distinct from $i$ such that~\eqref{eq:dress} 
holds with equality. 

\subsubsection*{Acknowledgments}
We thank Tadeusz Januszkiewicz, Alexander Lytchak, and 
Viktor Schroeder for helpful discussions. We acknowledge support from 
the Swiss National Science Foundation.


\bigskip\noindent
D.~Descombes ({\tt dominic.descombes@math.ethz.ch}),\\ 
U.~Lang ({\tt urs.lang@math.ethz.ch}),\\
Department of Mathematics, ETH Zurich, 8092 Zurich, Switzerland 


\end{document}